\documentclass[12pt,leqno]{amsart}
\usepackage{a4,latexsym,amssymb,amsfonts}
\usepackage{enumerate}
\usepackage{amsthm}
\numberwithin{equation}{section}
\newtheorem{theo}{Theorem}[section]
\newtheorem{lemma}[theo]{Lemma}
\newtheorem{prop}[theo]{Proposition}
\newtheorem{cor}[theo]{Corollary}
\newtheorem{defi}[theo]{Definition}
\theoremstyle{definition}
\newtheorem{rem}[theo]{Remark}
\newtheorem{exa}[theo]{Example}
\newtheorem{exs}[theo]{Examples}
\newcommand{\ds}{\displaystyle}
\newcommand{\Lp}{L^p ([-1,0],X)}
\newcommand{\xf}{\bigl(\begin{smallmatrix}
                           x\\
                           f
                       \end{smallmatrix}\bigr)}
\newcommand{\Wp}{W^{1,p}([-1,0],X)}
\newcommand{\sg}{$(T(t))_{t\geq 0}$ }
\newcommand{\sgs}{$(S(t))_{t\geq 0}$ }
\newcommand{\sgu}{$(U(t))_{t\geq 0}$ }
\newcommand{\SG}{$(\cT(t))_{t\geq 0}$ }
\newcommand{\SGO}{$(\cT_0 (t))_{t\geq 0}$ }
\def\RR{{\mathbb{R}}}
\def\CC{{\mathbb{C}}}
\def\NN{{\mathbb{N}}}

\def\cL{{\mathcal{L}}}

\def\cT{{\mathcal{T}}}
\def\cS{{\mathcal{S}}}

\def\cA{{\mathcal{A}}}
\def\cE{{\mathcal{E}}}
\def\cB{{\mathcal{B}}}
\def\cU{{\mathcal{U}}}

\def\sg{\(\displaystyle{(T(t))_{t\geq 0}}\) }

\begin{document}
\title{Semigroups and linear partial differential equations with delay}
\author{Andr\'as B\'atkai
 and Susanna Piazzera}\thanks{The authors gratefully
 acknowledge support by Konrad-Adenauer-Stiftung and Consiglio Nazionale
 delle Ricerche respectively.
 They also want to thank A. Rhandi (Marrakesh) and R. Schnaubelt
 (T\"ubingen)  for helpful discussions.}
\address{Mathematisches Institut, Universit\"at
T\"ubingen, Auf der Morgenstelle 10, 72076 T\"ubingen, Germany.}
\email{anba@michelangelo.mathematik.uni-tuebingen.de}
\email{supi@michelangelo.mathematik.uni-tuebingen.de}

\begin{abstract}
We prove the equivalence of the well-posedness of a partial differential 
equation with delay and an associated abstract Cauchy problem.
This is used to derive sufficient conditions for well-posedness, 
exponential stability and norm continuity of the solutions. 
Applications to a reaction-diffusion equation with delay are given.
\end{abstract}
\maketitle
\begin{section}{Introduction}
Partial differential equations with delay have been studied for many
years.
In an abstract way and using the standard notation (see
\cite{ku-sc} and \cite{wu})
they can be written as
\begin{equation}\tag*{\text{(DE)}}
\begin{cases}
        u'(t)=Au(t)+\Phi u_t ,&t\geq 0,\\
        u(0)=x, & \\
        u_0 =f,&
\end{cases}
\end{equation}
on a Banach space $X$, where $(A,D(A))$ is a (unbounded) linear operator on $X$
and the delay operator $\Phi$ is supposed to belong to $\cL (\Wp ,X)$.
J. Hale \cite{hale} and G. Webb \cite{webb} were among the first who applied 
semigroup theory to the study of such equations, and we refer to \cite{rhandi}
and
\cite{wu} for more recent references.
As a  first step one has to choose an appropriate state space.
One of the possibilities is to work in the space of continuous $X$-valued functions.
In this case, the relationship between solutions of (DE) and a corresponding semigroup  has
been studied widely (see for example
\cite[Section VI.6]{en-na}) and is well understood.
On the other hand, the state space $\cE :=X\times\Lp$ turns
out to be a better choice with regards to certain applications 
(to control theory, see \cite{nak1}) and will be used in this paper.
We will show in Section 2 that the linear partial differential equation with delay is equivalent to an abstract Cauchy
problem
\begin{equation}\tag*{\text{(ACP)}}
\begin{cases}
\cU '(t)=\cA\,\cU (t),\qquad t\geq 0,\\
\cU (0)=\xf,
\end{cases}
\end{equation}
on the space $\cE$. Similar results for neutral differential
equations on finite dimensional spaces $X$ were proved by F. Kappel and K. Zhang
\cite{ka-zh} (see also \cite{bu-he-st}).
In the third section we give sufficient conditions such that the
operator $(\cA ,D(\cA ))$ generates a strongly continuous
semigroup on the space $\cE$. Assuming that $(A,D(A))$ is the generator of a
strongly continuous semigroup on $X$, we show that the operator $\cA$ is given by the 
sum $\cA_0 +\cB$, where $(\cA_0 ,D(\cA_0 ))$ generates a strongly continuous
semigroup on $\cE$ and $\cB$ is $\cA_0$-bounded. Moreover, we give a condition
 on
the operators $A$ and $\Phi$ such that $\cB$ is a Miyadera perturbation of
$\cA_0$ and thus $\cA$ generates a strongly continuous semigroup on $\cE$.
Finally, for $1<p<\infty$, we show that if the operator $\Phi$ is given by the Riemann-Stieltjes
integral of a function of bounded variation $\eta :[-1,0]\longrightarrow \cL
(X)$, then $\cB$ is a Miyadera perturbation of $\cA_0$ for every generator $(A,D(A))$ 
and $\cA=\cA_0 +\cB$ generates a strongly
continuous semigroup on $\cE$. For similar equations the well-posedness was 
proved, using a completely different setting, by J. Pr\"u\ss\  in  \cite[Theorem I.1.2]{pruss}.
This part was inspired, on the one side, by a paper of K. Kunisch
and W. Schappacher \cite{ku-sc} containing necessary conditions such
that $\cA$ generates a strongly continuous semigroup on $\cE$, on
the other side, by a work of K.-J. Engel \cite{en-habil}, who used the theory
of unbounded operator matrices to treat the
case of bounded $A$.
In the fourth section, we prove that  the stability of the equation without
delay persists in the equation with delay.
We characterize the resolvent and use the theorem of Gearhart.
This part is a generalization of a recent result of A. Fischer and J. van Neerven
\cite{FvN}. In \cite[Theorem 5.1.7]{CZ} one finds a similar stability result for the finite
dimensional case.
In the last section we give a sufficient condition for the eventual 
norm continuity
of the semigroup generated by the operator $(\cA ,D(\cA ))$. We show that if
$(A,D(A))$ generates an immediately norm continuous semigroup on $X$ and
$\Phi$ satisfies a technical condition, 
then $(\cA ,D(\cA ))$ generates a strongly
 continuous semigroup which is norm continuous for $t>1$. We use these results to obtain exponential stability of the solutions.
 
Finally, we illustrate our results on a reaction-diffusion equation with delay. 
\end{section}
\begin{section}{The semigroup approach}
Consider the equation
\begin{equation}\tag*{\text{(DE)}}
\begin{cases}
        u'(t)=Au(t)+\Phi u_t ,&t\geq 0,\\
        u(0)=x, & \\
        u_0 =f,&
\end{cases}
\end{equation}
where
\begin{itemize}
\item $x\in X$, $X$ is a Banach space,
\item $A:D(A)\subseteq X\longrightarrow X$ is a closed and
densely defined linear operator,
\item $f\in L^p
([-1,0],X)$, $1\leq p <\infty$,
\item  $\Phi:W^{1,p}([-1,0],X)\longrightarrow X$ is a bounded linear 
operator,
\item $u:[-1,\infty )\longrightarrow X$ and $u_t
:[-1,0]\longrightarrow X$ is defined by $u_t (\sigma ):=u(t+\sigma )$ for 
$\sigma\in[-1,0]$.
\end{itemize}
We say that a function $u:[-1,\infty )\longrightarrow X$ is
 a {\it (classical) solution} of (DE) if
\begin{enumerate}[(i)]
\item $u\in C([-1,\infty ),X)\cap C^1 ([0,\infty ),X)$,
\item $u(t)\in D(A)$ and $u_t \in \Wp$
for all $t\geq 0$,
\item $u$ satisfies (DE) for all $t\geq 0$.
\end{enumerate}
It is now our purpose to investigate existence and uniqueness of the
solutions of (DE). To do this we introduce the Banach space
$$\cE :=X\times\Lp $$
and the operator
\begin{equation}\label{delaymatrix}
\cA :=\begin{pmatrix}
             A&\Phi\\
             0&\frac{d}{d\sigma}
        \end{pmatrix}
\end{equation}
with domain
\begin{equation}\label{delaydomain}
D(\cA ):=\left\{\xf\in D(A)\times
W^{1,p}([-1,0],X)\ :\ f(0)=x\right\}.
\end{equation}
\begin{lemma}\label{Aclosed}
The operator $(\cA ,D(\cA ))$ is closed and densely defined.
\end{lemma}
The proof of Lemma \ref{Aclosed} is straightforward and is omitted.
\begin{rem}
A necessary condition for (DE) to have a solution is that
$u_0 =f\in\Wp$ and $u(0)=x\in D(A)$, i.e. $\xf\in D(\cA )$.
\end{rem}
In the following, we will see that equation (DE) and the abstract Cauchy
problem associated to the operator $(\cA ,D(\cA ))$
$$\text{(ACP)}\qquad\begin{cases}
         \dot{\cU}(t)=\cA \,\cU (t),& t\geq 0,\\
     \cU (0)=\xf ,&
     \end{cases}$$
are ``equivalent",
i.e., (DE) has a unique solution for every $\xf\in D(\cA )$
 and the solutions depend continuously on the initial values
if and only if (ACP) is well-posed (in the sense of \cite[Definition II.6.8]{en-na}).
\begin{prop}\label{uniqueness}
Let $\xf\in D(\cA )$ and let $u:[-1,\infty )\longrightarrow X$ be a solution of
 (DE). Then the map
\[\RR_+ \ni t\mapsto \bigl(\begin{smallmatrix}
                               u(t)\\
                                u_t
                \end{smallmatrix}\bigr)\in\cE\]
is a classical solution of the abstract Cauchy problem (ACP) associated
to the
operator $(\cA ,D(\cA ))$ with initial value $\xf$.
\end{prop}
\begin{proof}
Since the function $u$ is a solution of (DE), we have 
$u\in C^1 ([0,\infty ),X)$,
$u(t)\in D(A)$, $u_t \in\Wp$ and $\dot{u}(t)=Au(t)+\Phi u_t$ for all
$t\geq 0$, $u(0)=x$, and
finally $u_0 =f$.
It remains to
show that the map
$t\mapsto u_t$ is continuously
differentiable for $t\geq 0$ and that
$\frac{d}{dt}u_t =\frac{d}{d\sigma } u_t$ for
all $t\geq 0$. So let $T>t\geq 0$. Then the function $u_{|[-1,T]}$
can be extended to a function $v\in W^{1,p}(\RR ,X)$, i.e., $v$ is in the 
domain
of the first derivative, which is the generator of the left shift semigroup  
on $L^p (\RR ,X)$. From the definition of generator we have
$$\frac{d}{dt}v(t+\cdot )=\frac{d}{d\sigma}v(t+\cdot )$$
in $L^p (\RR ,X)$ for $t\geq 0$. This implies that
$$\frac{d}{dt}u_t =\frac{d}{dt}u(t+\cdot )=\frac{d}{d\sigma
}u(t+\cdot )=
\frac{d}{d\sigma }u_t$$
in $\Lp$ and that 
the map $\RR_+ \ni t\mapsto \frac{d}{d\sigma }u_t \in\Lp$ is continuous.
\end{proof}
We now prove the converse to Proposition \ref{uniqueness}. Similar results with different conditions on $A$ and $\Phi$ were proved by G. Webb in \cite{webb1}.
\begin{prop}\label{abschnitte}
Let $\xf\in D(\cA )$ and $\cU :[0,\infty )\longrightarrow \cE$,
$\cU (t)=\bigl(\begin{smallmatrix}
z(t)\\
v(t)
\end{smallmatrix}\bigr)$, be a classical solution of the abstract 
Cauchy problem (ACP) with
initial value $\xf$ associated to the 
operator $(\cA ,D(\cA ))$. Let
$u:[-1,\infty )\longrightarrow X$ be the function defined by
\begin{equation}\label{solution}
u(t):=\begin{cases}
      z(t) ,&   t\geq 0,\\
      f(t),&   t\in[-1,0).
      \end{cases}
\end{equation}
Then $u_t =v(t)$ for every $t\geq 0$ and $u$ is a solution of
(DE).
\end{prop}
\begin{proof}
Since $\cU$ is a classical solution of (ACP), 
 $v\in C^1\left(\RR_+,L^p([-1,0],X)\right)$ solves the Cauchy problem
\begin{equation}\label{partial}
\begin{cases}
       \frac{d}{dt}v(t)=
       \frac{d}{d\sigma }v(t),&t\geq 0,\\
       v(t)(0)=z(t),& t\geq 0,\\
       v(0)=f&
\end{cases}
\end{equation}
in the space $\Lp$. We now observe that the map
$\RR_+ \ni t\mapsto u_t \in\Lp$ solves (\ref{partial}) in the space $\Lp$.  
So let $w(t):=u_t -v(t)$
for $t\geq 0$. Then $w$ is a classical solution of the Cauchy
problem
\begin{equation}\label{partial1}
\begin{cases}
       \frac{d}{dt}w(t)=
       \frac{d}{d\sigma }w(t),& t\geq 0,\\
       w(t)(0)=0,& t\geq 0,\\
       w(0)=0.&
\end{cases}
\end{equation}
Since (\ref{partial1}) is the abstract Cauchy problem associated to the
generator of the nilpotent left shift semigroup on $\Lp$
with initial value $0$, we have that $w(t)= 0$ for all $t\geq 0$. 
Therefore $\cU (t)=\bigl(\begin{smallmatrix}
u(t)\\
u_t
\end{smallmatrix}\bigr)$ for all $t\geq 0$ and $u$ is a solution
of (DE).
\end{proof}
Let $\pi_1 :\cE\longrightarrow X$ be the projection onto the first
component of $\cE$, i.e., $\pi_1 \xf :=x$ for all $\xf\in\cE$.
\begin{cor}\label{existenz}
If $(\cA ,D(\cA ))$ is the generator of a strongly continuous semigroup \SG
on $\cE$, then equation (DE) has a
unique solution $u$ for every $\xf\in D(\cA )$, which is given by
\begin{equation}\label{solution.sem}
u(t)=
\begin{cases}
\pi_1 \bigl(\cT (t)\xf\bigr), & t\geq 0,\\
f(t), & t\in [-1,0).
\end{cases}
\end{equation}
\end{cor}
\begin{proof}
Since $(\cA ,D(\cA ))$ is the generator of the strongly continuous
semigroup \SG , the Cauchy problem (ACP) has a unique classical
solution $\cU$ for all $\xf\in D(\cA )$ and
$$\cU (t)=\cT (t)\xf ,\qquad t\geq 0.$$
Then the function $u$ defined in
(\ref{solution.sem}) is a solution of (DE) by Proposition \ref{abschnitte}. 
Uniqueness follows from
Proposition \ref{uniqueness}.
\end{proof}
\begin{cor}
If $(\cA ,D(\cA ))$ is the generator of a strongly continuous semigroup 
$(\cT (t))_{t\geq 0}$, the 
function $u:[-1,\infty )\longrightarrow X$
defined by (\ref{solution.sem}) for a given $\xf\in\cE$ satisfies the integral 
equation
\begin{equation}\label{mildsolution}
u(t)=
\begin{cases}
x+A\int_0^t u(s)\,ds+\Phi \int_0^t u_s \,ds,\qquad &t\geq 0,\\
f(t), & \text{a.e. }t\in[-1,0).
\end{cases}
\end{equation}
\end{cor}
\begin{proof}
Let $\pi_2 :\cE\longrightarrow \Lp$ be the
projection onto the second component of $\cE$, i.e., $\pi_2 \xf :=f$ for all
$\xf\in\cE$.\\
Then, by Proposition \ref{abschnitte} and using the density of $D(\cA )$
in $\cE$, we obtain $u_t =\pi_2 \bigl(\cT (t)\xf \bigr)$ for all $\xf\in\cE$ 
and $t\geq
0$.
Take now  the first component of the identity
$$\cT (t)\xf -\xf =\cA \int_0^t \cT (s)\xf\,ds, \qquad t\geq 0,$$
to obtain (\ref{mildsolution}).
\end{proof}
A solution of the integral equation (\ref{mildsolution}) is called 
{\it mild solution} in the literature, see for example \cite[Section 2.2]{nak2}.
At this point we need the appropriate terminology.
\begin{defi}
We call (DE) {\bf well-posed} if
\begin{enumerate}[(i)]
\item for every $\xf\in D(\cA )$ there is a unique solution $u(x,f,\cdot )$ 
of (DE) and
\item the solutions depend continuously on the initial values, i.e.,
if a sequence $\bigl(\begin{smallmatrix}
                                               x_n\\
                                               f_n
                                          \end{smallmatrix}\bigr)$
in $D(\cA )$ converges to $\xf\in D(\cA )$, then $u(x_n ,f_n ,t)$
converges
to $u(x,f,t)$ uniformly for $t$ in compact intervals.
\end{enumerate}
\end{defi}
This well-posedness of (DE) can now be characterized by the well-posedness of
the abstract Cauchy problem (ACP) for the operator $(\cA ,D(\cA ))$. 
\begin{theo}\label{caratterizzazione}
Let $(\cA ,D(\cA ))$ be the operator defined by (\ref{delaymatrix}) and
(\ref{delaydomain}). Then the following assertions are equivalent.
\begin{enumerate}[(i)]
\item (DE) is well-posed.
\item $(\cA ,D(\cA ))$ is the generator of a strongly continuous semigroup
on $\cE$.
\end{enumerate}
\end{theo}
\begin{proof}
We first show $(i)\Rightarrow (ii)$. Assume that for every $\xf\in D(\cA )$
equation (DE) has a unique solution $u$. Then Proposition
\ref{uniqueness} yields that for every $\xf\in D(\cA )$ the abstract 
Cauchy problem (ACP) has a classical solution which is unique
by Proposition \ref{abschnitte}. It is easy to see that these solutions depend
continuously on the initial values. 
Finally, by Lemma \ref{Aclosed}, $(\cA ,D(\cA ))$ is a closed and densely defined operator.
So $(\cA,D(\cA ))$ generates a strongly continuous semigroup
on $\cE$ by \cite[Theorem II.6.7]{en-na}.
Conversely, if $\cA$ is a generator, we have by Corollary \ref{existenz}
that for every initial value
$\xf\in D(\cA )$ there is a unique solution $u$ of
(DE) which is given by (\ref{solution.sem}). This implies that the
solutions
depend continuously on the initial values.
\end{proof}
\end{section}
\begin{section}{The generator property}
In the previous section we transformed the problem of solving the partial 
differential equation with delay 
(DE) into the functional analytical problem: When does $\cA$ generate a strongly
 continuous semigroup on $\cE$?
 
In the following, we will give sufficient conditions on $A$ and $\Phi$ 
such that this is true. 
First, we observe that
we can write $\cA$ as the sum $\cA_0 +\cB$, where
\begin{equation}\label{delaymatrix0}
\cA_0 :=\begin{pmatrix}
             A&0\\
             0&\frac{d}{d\sigma}
          \end{pmatrix}
\end{equation}
with domain
\begin{equation}\label{delaydomain0}
D(\cA_0 ):=D(\cA )=\left\{ \xf\in D(A)\times \Wp\ :\ f(0)=x\right\}
\end{equation}
and 
\[\cB :=\begin{pmatrix}
             0&\Phi\\
             0&0
        \end{pmatrix}\in\cL (D(\cA_0 ),\cE ).\]
The idea now is to show first that under appropriate conditions $\cA_0$ becomes a
generator and then apply perturbation results to show that the
sum $\cA_0 +\cB$ is a generator as well. The first step is quite easy.
\begin{prop}
Let $(A,D(A))$ be the generator of a strongly continuous semigroup \sgs on $X$.
Then
$(\cA_0 ,D(\cA_0 ))$ generates the strongly continuous semigroup
$(\cT_0 (t))_{t\geq 0}$ on $\cE$ given by
\begin{equation}\label{unperturbeddelaysemigroup}
\cT_0 (t):=\begin{pmatrix}
             S(t)&0\\
             S_t &T_0 (t)
        \end{pmatrix},
\end{equation}
where $(T_0 (t))_{t\geq 0}$ is the nilpotent left shift semigroup on
$L^p ([-1,0],X)$ and $S_t :X\rightarrow L^p ([-1,0],X)$ is
defined by
\[(S_t \,x)(\tau ):=\begin{cases}
                              S(t+\tau )x, & -t<\tau\leq 0,\\
                              0,                &-1\leq\tau \leq -t.
                              \end{cases}\]
\end{prop}
Therefore, we will always assume that
\[(A,D(A))\text{ {\it generates a strongly continuous semigroup
\sgs on }}X.\]
We will see later in Remark \ref{necessary} that this condition is necessary for the well-posedness in case of many applications .
To the perturbation $\cB$ we will apply the theorem of Miyadera-Voigt (see \cite{miy} and \cite{voigt1}), which we quote from \cite[Corollary III.3.16]{en-na}.
\begin{theo}\label{theo.miyadera-voigt}
Let $(G,D(G))$ be the generator of a strongly continuous semigroup
\sg on a Banach
space $X$ and let $C\in\cL ((D(G),\|\cdot\|_G ),X)$ satisfy
\begin{equation}\label{m-v}
\int_0^{t_0}\|CT(r)x\|\,dr\leq q\|x\|\qquad\mbox{for all }x\in D(G)
\end{equation}
and some $t_0>0$, $0\leq q<1$. Then $(G+C,D(G))$ generates a strongly continuous
semigroup \sgu on $X$ which satisfies
$$U(t)x=T(t)x+\int_0^t T(t-s)CU(s)x\,ds\qquad\mbox{and}$$
$$\int_0^{t_0}\|CU(t)x\|\,dt\leq \frac{q}{1-q}\|x\|\qquad\mbox{for }x\in
D(G)\mbox{ and }t\geq 0.$$
\end{theo}
In our situation, $(\cA ,D(\cA ))$ generates a
strongly continuous semigroup on $\cE$ if
there exist $t_0 >0$ and $0\leq q<1$ such that
\begin{equation}\label{miyadera.puro}
\int_0^{t_0} \left\|\cB\,\cT_0 (r)\xf\right\|\,dr
\leq q \left\|\xf
\right\| 
\end{equation}
for all $\xf\in D(\cA_0 )$. However, since
\begin{equation*}
\begin{split}
\int_0^{t_0} \left\|\cB\,\cT_0 (r)\xf\right\|\,dr &=
\int_0^{t_0} \left\| \begin{pmatrix}
             0&\Phi\\
             0&0
        \end{pmatrix}\,
\begin{pmatrix}
             S(r)&0\\
             S_r &T_0 (r)
        \end{pmatrix}\,\begin{pmatrix}
                         x\\
                         f
                       \end{pmatrix}
\right\| \, dr\\
&= \int_0^{t_0} \|\Phi (S_r x + T_0 (r)f) \|\,dr\,,
\end{split}
\end{equation*}
we conclude that (\ref{miyadera.puro}) holds if and only if there exist 
$t_0 >0$ and $0\leq q<1$ such that
\begin{equation}\tag*{(M)}
\qquad
\int_0^{t_0} \|\Phi (S_r \,x + T_0 (r)f) \|\,dr
\leq q\,\left\|\xf\right\|
\end{equation}
for all $\xf\in D(\cA_0 )$. We therefore obtain the following result.
\begin{theo}
Let $(A,D(A))$ be the generator of a strongly continuous semigroup on $X$
and let condition (M) be satisfied. Then
the operator $(\cA ,D(\cA ))$ is the generator of a strongly
continuous semigroup on $\cE$. Thus, (DE) is well-posed.
\end{theo}
We now give some important examles of $\Phi$ to satisfy this condition (M).
\begin{exs}\label{esempi}
\begin{enumerate}[(a)]
\item Let $\Phi$ be  bounded from $L^p ([-1,0],X)$ to $X$. Then
the perturbation $\cB$ is bounded, and $(\cA ,D(\cA
))$ is a generator on $\cE$.
\item Let $1<p<\infty$ and let $\eta :[-1,0]\rightarrow \cL (X)$
be of bounded variation. 
Let $\Phi :C([-1,0],X)\rightarrow X$ be the bounded linear
operator given by the Riemann-Stieltjes integral
\begin{equation}\label{phi}
\Phi
(f):=\int_{-1}^0 d\eta\,f \qquad\mbox{for all } f\in C([-1,0],X).
\end{equation}
Since $W^{1,p}([-1,0],X)$ is
continuously embedded in $C([-1,0],X)$, $\Phi$ defines a bounded
operator from $W^{1,p}([-1,0],X)$ to $X$. For
$0<t<1$ we obtain that
\begin{equation*}
\begin{split}
\int_0^t \|\Phi (S_r \,x + T_0 (r)f) \|\,dr
&=\int_0^t \left\|\int_{-1}^{-r}d\eta (\sigma )\,f(r+\sigma )+
   \int_{-r}^0 d\eta (\sigma )\,S(r+\sigma )x\right\|\,dr\\
&\leq \int_0^t \int_{-1}^{-r}\|f(r+\sigma )\|\,d|\eta |(\sigma )\,dr+
   \int_0^t \int_{-r}^0 \|S(r+\sigma )x\|\,d|\eta |(\sigma )\,dr\\
&\leq \int_{-t}^0\int_{\sigma}^0 \|f(s)\|\,ds\,d|\eta |(\sigma )+
   \int_{-1}^{-t}\int_{\sigma}^{t+\sigma}\|f(s)\|\,ds\,d|\eta |(\sigma )\\
&\quad   +\int_0^t M\|x\|  |\eta |([-1,0])\,dr\\
&\leq \int_{-t}^0 (-\sigma )^{1/p'}\|f\|_p \,d|\eta |(\sigma )+
   \int_{-1}^{-t} t^{1/p'}\|f\|_p \,d|\eta |(\sigma )\\
&\quad +tM\|x\|  |\eta |([-1,0])\\
&\leq \int_{-1}^0 t^{1/p'}\|f\|_p \,d|\eta |(\sigma )+
   tM\|x\|  |\eta |([-1,0])\\
&=(t^{1/p'}\|f\|_p +tM\|x\| )\, |\eta |([-1,0]),
\end{split}
\end{equation*}
where $\frac{1}{p}+\frac{1}{p'}=1$, \(\ds{M:=\sup_{r\in [0,1]}\|S(r)\|}\)
and $|\eta |$ is the
positive Borel measure on $[-1,0]$ defined by the total variation of $\eta$.
Finally we conclude that
\begin{equation}\label{cond.miyadera}
\int_0^t \|\Phi (S_r x + T_0 (r)f) \|\,dr
\leq t^{1/p'}M\, |\eta |([-1,0])(\|f\|_p +\|x\| )
\end{equation}
for all $0<t<1$. Choose now $t_0$ small enough such that
$t_0^{1/p'}M\, |\eta |([-1,0])<1$. Then condition (M) is
satisfied with $q:=t_0^{1/p'}M\,|\eta |([-1,0])$.
\item An important special case of (b) are the operators $\Phi$ defined by
\(\ds{\Phi (f):=\sum_{k=0}^{n}B_k f(h_k )}\), $f\in\Wp$, where $B_k \in\cL
(X)$ and $h_k \in [-1,0]$ for $k=0,\ldots ,n$.
\end{enumerate}
\end{exs}
\begin{rem}\label{necessary}
We note that from the perturbation theorem of Miyadera-Voigt (see Theorem \ref{theo.miyadera-voigt}) 
it follows that the generator property of $(A,(D(A))$ is necessary and 
sufficient for the well-posedness of (DE) if $\Phi$ is defined as in 
(\ref{phi}) because we can choose $q$ small enough. 
Similar results for more special cases were proved by Kunisch and Schappacher (see \cite[Proposition 4.2]{ku-sc}) and for a slightly different equation by Pr\"u\ss \  (see \cite[Corollary I.1.4]{pruss}).
\end{rem}
\end{section}
\begin{section}{Stability}
In this section, we prove a stability result for the delay equation
(DE) extending a recent result  of Fischer and van
 Neerven \cite{FvN}. First, we recall some notations and definitions.
 
Let $\cT:=\left(T(t)\right)_{t\geq 0}$ be a $C_0$-semigroup of bounded
linear operators on the Banach space $X$ with generator $(G,D(G))$.
%We define the following quantities.
%
%
The {\it spectral bound} of $G$ is given by
$$s(G):=\left\{\Re\lambda\,:\,\lambda\in\sigma(G)\right\},$$
the {\it abscissa of uniform boundedness} of the resolvent of $G$ is defined by
$$s_0 (G):=\inf\left\{\omega\in\RR\,:\, \{\Re\lambda>\omega\}
\subset\varrho(G)\ \mbox{and}\ \sup_{\Re\lambda>\omega}
\|R(\lambda,G)\|<\infty\right\},$$
and {\it the uniform growth bound} or {\it type} of the semigroup 
$$\omega_0 (G):=\inf\left\{\omega\in \RR\,:\,\exists M>0\
\mbox{such that}\ \|T(t)\|\leq Me^{\omega t}\ \forall t\geq 0\right\}.$$
We say that $\cT$ is {\it uniformly exponentially stable} if
$\omega_0 (G)<0$. It is known (see \cite[Sections 1.2, 4.1]{vN}) that
$$-\infty\leq s(G)\leq s_0 (G)\leq \omega_0 (G)<\infty.$$
The theorem of Gearhart (see \cite[Theorem V.1.11]{en-na}) says that in
 a Hilbert space $X$
\begin{equation}\label{gearhart}
s_0(G)=\omega_0(G).
\end{equation}
In order to find estimates for the above quantities, we
first calculate the resolvent $R(\lambda ,\cA )$ and the
resolvent set $\rho (\cA )$ of the operator $\cA$. Given
$\lambda\in\CC$ and
$\bigl(\begin{smallmatrix}
              y\\
              g
\end{smallmatrix}\bigr)\in \cE$ we are looking for $\xf\in D(\cA )$
 such that
\begin{equation*}
(\lambda -\cA )\xf
=\begin{pmatrix}
 (\lambda -A)x -\Phi f\\
  \lambda f-f'
\end{pmatrix}=\begin{pmatrix}
                      y\\
                      g
                \end{pmatrix}.
\end{equation*}
Since $f(0)=x$, the second component of this identity is equivalent to
\begin{equation}\label{risolvente1}
 f=\epsilon_{\lambda }\otimes x +R(\lambda ,A_0 )g,
\end{equation}
where $\epsilon_{\lambda }(s):=e^{\lambda s}$ for $s\in [-1,0]$
and $(A_0 ,D(A_0 ))$ is the infinitesimal generator of the
nilpotent left shift semigroup $(T_0 (t))_{t\geq 0}$ on
$L^p ([-1,0],X)$ whose spectrum is empty.
Hence, $x$ has to satisfy the equation
\begin{equation}\label{risolvente2}
(\lambda -A -\Phi (\epsilon_{\lambda }\otimes Id))x=\Phi
R(\lambda ,A_0 )g+y.
\end{equation}
This leads to the following lemma (see also \cite{en-habil} and
\cite{nak2}).
\begin{lemma}\label{spectrum}
For $\lambda\in\CC$ we have $\lambda\in\rho (\cA )$ if and only if
$\lambda\in\rho
(A+\Phi (\epsilon_{\lambda }\otimes Id))$. Moreover, for
$\lambda\in\rho (\cA )$
the resolvent $R(\lambda ,\cA )$ is given by
\begin{equation}\label{risolvente}
\begin{pmatrix}
R(\lambda ,A+\Phi (\epsilon_{\lambda }\otimes Id)) &
R(\lambda ,A+\Phi (\epsilon_{\lambda }\otimes Id))\Phi R(\lambda
,A_0 )\\
\epsilon_{\lambda }\otimes
R(\lambda ,A+\Phi (\epsilon_{\lambda }\otimes Id))&
[\epsilon_{\lambda }\otimes
R(\lambda ,A+\Phi (\epsilon_{\lambda }\otimes Id))\Phi +Id]R(\lambda
,A_0 )
\end{pmatrix}.
\end{equation}
\end{lemma}
\begin{proof}
Let $\lambda\in\rho (A+\Phi (\epsilon_{\lambda }\otimes Id))$.
Then the matrix in (\ref{risolvente}) is a bounded
operator from $\cE$ to $D(\cA )$  defining the inverse of $(\lambda -\cA )$.
Conversely, if $\lambda\in\rho (\cA )$, then for every
$\bigl(\begin{smallmatrix}
            y\\
            g
         \end{smallmatrix}\bigr)\in\cE$ there exists a unique
$\xf\in D(\cA )$ such that (\ref{risolvente1}) and (\ref{risolvente2})
hold. In particular, for $g=0$ and for every $y\in X$, there exists a
unique $x\in D(A)$ such that
$$(\lambda -A -\Phi (\epsilon_{\lambda }\otimes Id))x=y.$$
This means that $(\lambda -A -\Phi (\epsilon_{\lambda }\otimes Id))$ is
invertible, i.e.
$\lambda\in\rho (A-\Phi (\epsilon_{\lambda }\otimes Id))$.
\end{proof}
We now assume the well-posedness conditions from the
previous section. In particular, we will assume that $(A,D(A))$ generates a
strongly continuous semigroup on $X$, that $p>1$ and that $\Phi$
satisfies the assumption (\ref{phi})  of Example \ref{esempi} (b), i.e.
$$\Phi f=\int_{-1}^0 d\eta\,f ,$$
where $\eta :[-1,0]\longrightarrow \cL (X)$ is a function of bounded
variation.
Hence, our matrix $(\cA ,D(\cA ))$, defined in (\ref{delaymatrix})
and (\ref{delaydomain}), is the generator of a strongly
continuous semigroup \SG on the Banach space $\cE$ and (DE)
is well-posed.
For this generator  we can now estimate $s_0(\cA)$, generalizing \cite[Theorem 3.3]{FvN} with a similar proof.
\begin{theo}\label{cond.stabilita}
Assume that $s_0 (A)<0$ and let $\alpha\in(s_0 (A),0]$. If
\begin{equation}
\sup_{\omega\in\RR}\left\|\Phi (\epsilon_{\alpha +i\omega}\otimes Id)
\right\|<\frac1{\sup_{\omega\in\RR}\left\|R(\alpha +i\omega ,A)\right\|},
\end{equation}
then $s_0 (\cA )<\alpha\leq 0$.
\end{theo}
To prove this statement we recall from \cite[Proposition 1.1]{FvN} the
following lemma in a modified form (compare also \cite[Theorem IV.1.16]{kato}).
\begin{lemma}\label{lemmadelta}
Let  $A$ be a closed linear
operator on a Banach space $X$, and
suppose $\lambda\in\rho (A)$. If $\Delta\in\cL (X)$ satisfies
$$\|\Delta\|\leq(1-\delta )\frac{1}{\left\|R(\lambda ,A)\right\|}$$
for some $\delta\in (0,1)$, then $\lambda\in\rho (A+\Delta )$ and
$$\left\|R(\lambda ,A+\Delta )\right\|\leq \frac{1}{\delta }\left\|R(\lambda ,A)\right\|.$$
\end{lemma}
\begin{proof}[Proof of Theorem \ref{cond.stabilita}]
Choose $\delta\in (0,1)$ such that
$$\sup_{\omega\in\RR}\left\|\Phi (\epsilon_{\alpha +i\omega }\otimes Id)
\right\|
\leq (1-\delta ) \frac{1}{\sup_{\omega\in\RR}
\left\|R(\alpha +i\omega ,A)\right\|}.$$
We observe that $\Phi (\epsilon_{\lambda}\otimes Id)$ is an analytic
function and that the  suprema of bounded analytic functions along
vertical lines, $\Re\lambda =c$,  decrease as $c$ increases
(see \cite[Chapter 6.4]{conway} for the details).
For all $\lambda\in\CC$ with $\Re\lambda >\alpha$ we obtain
$$\left\|\Phi(\epsilon_{\lambda }\otimes Id)\right\|\leq
\sup_{\omega\in\RR}\left\|\Phi (\epsilon_{\alpha +i\omega}\otimes Id)
\right\|\leq \frac{(1-\delta)}{\sup_{\omega\in\RR}
\left\|R(\alpha +i\omega ,A)\right\|}
\leq \frac{(1-\delta)}{\left\|R(\lambda,A)\right\|}.$$
Therefore, by Lemma \ref{lemmadelta},
$\left\{\Re\lambda >\alpha\right\}\subset\rho
\left(A+\Phi (\epsilon_{\lambda}\otimes Id)\right)$,
and for all $\lambda\in\CC$ with $\Re\lambda >\alpha$ we have
\begin{equation*}
\left\|R\left(\lambda ,A+\Phi (\epsilon_{\lambda}\otimes Id)\right)
\right\| \leq 
\frac{1}{\delta} \|R(\lambda ,A)\|.
\end{equation*}
Hence, by Lemma \ref{spectrum} ,
$$\{\Re \lambda >\alpha\}\subset\rho (\cA ).$$
We show that
$R(\lambda ,\cA )$ is bounded on this halfplane. We have shown above
that $R(\lambda ,A+\Phi (\epsilon_{\lambda}\otimes Id))$ is bounded.
The operator $A_0$ generates the nilpotent left shift, so
$R(\lambda ,A_0 )$ is bounded on this right halfplane.
The function $\Phi (\epsilon_{\lambda}\otimes Id)$ is  bounded, analytic
and $\Phi R(\lambda , A_0 )$ is continuous and bounded, since
\begin{equation*}
\begin{split}
\|\Phi R(\lambda ,A_0 )f\| &=
\left\|\int_{-1}^0 d\eta (\sigma )\int_{\sigma }^0
   e^{\lambda (\sigma -\tau )}f(\tau )\,d\tau\right\|\\
&\leq \int_{-1}^0 \int_{\sigma }^0
   \|e^{\lambda (\sigma -\tau )}f(\tau )\|\,d\tau\,d|\eta |(\sigma )\\
&=\int_{-1}^0 \int_{\sigma }^0
   e^{\Re\lambda (\sigma -\tau )}\|f(\tau )\|\,d\tau\,d|\eta |(\sigma )\\
&\leq\int_{-1}^0 \int_{\sigma }^0
   e^{-\alpha}\|f(\tau )\|\,d\tau\,d|\eta |(\sigma )\\
&\leq\int_{-1}^0 \int_{-1}^0
   e^{-\alpha }\|f(\tau )\|\,d\tau\,d|\eta |(\sigma )\\
&\leq e^{-\alpha }\,|\eta | ([-1,0])\,\|f\|_p
\end{split}
\end{equation*}
for every $\lambda$ with $\Re\lambda >\alpha$ and every $f\in\Lp$.
Therefore $\sup_{\Re\lambda >\alpha}\|R(\lambda ,\cA )\|<\infty $,
and we  conclude that $s_0 (\cA )<\alpha\leq 0$.\end{proof}
Using Gearhart's theorem quoted above, this result can be improved for
Hilbert spaces.
\begin{cor}\label{stab.hilbert}
Take $p=2$ and $X$  a Hilbert space. If $\omega_0 (A)<\alpha\leq 0$
and
$$\sup_{\omega\in\RR}\left\|\Phi (\epsilon_{\alpha +i\omega}\otimes Id)
\right\|<\frac{1}{\sup_{\omega\in\RR}\left\|R(\alpha +i\omega ,A)\right\|},
$$
then $\omega_0 (\cA )<\alpha\leq 0$.
\end{cor}
\begin{exa}\label{parab}
We consider the reaction diffusion equation with delay (see
\cite[Section 2.1]{wu})
\begin{equation*}
\begin{cases}
        &\partial_tw(x,t)=\Delta w(x,t)+c\int_{-1}^0 w(x,t+\tau)dg(\tau),
	\quad x\in\Omega,\ t\geq0,\\
        & w(x,t)=0, \quad x\in \partial\Omega,\ t\geq0, \\
        & w(x,t)=f(x,t), \quad (x,t)\in\Omega\times[-1,0],
\end{cases}
\end{equation*}
where $c$ is a constant, $\Omega\subset\RR^n$ a bounded domain, $f(\cdot,t)\in L^2(\Omega)$
for all $t\geq0$, $f(\cdot,0)\in W^{1,2}_0(\Omega)\cap W^{2,2}(\Omega)$ and the map
$[-1,0]\ni t\mapsto f(\cdot,t)\in L^2(\Omega)$ belongs to $W^{1,2}\left([-1,0],L^2(\Omega)\right)$.
The function $g:[-1,0]\to[0,1]$ is the Cantor function (see \cite[Example I.8.15]{ge-ol}),
which is singular and has total variation $1$. 
We consider $X:=L^2(\Omega)$, $A:=\Delta_D$ the Dirichlet-Laplacian with usual domain and
$\eta:=c\cdot g\cdot Id$. The well-posedness follows from the calculations in Examples \ref{esempi}(b).
We have to verify the stability
estimate. First, the expression with $\Phi$ satisfies
$$
\left\|\Phi (\epsilon_{\lambda}\otimes Id)y\right\|=|c|\|y\|\left\|\int_{-1}^0 e^{\lambda\tau}dg(\tau)\right\| \leq |c|\|y\|.
$$
The other expression, using that $A$ is a normal operator on a Hilbert 
space (see \cite[Section V.3.8]{kato}, can be computed as
$$
\sup_{\omega\in\RR}\left\|R(i\omega ,A)\right\| = \sup_{\omega\in\RR}\frac1{d\left(i\omega,\sigma(A)\right)} = \frac1{d\left(0,\sigma(A)\right)} =\frac1{|\lambda_1|},
$$
where $\lambda_1$ is the first eigenvalue of the Laplacian.
Thus the solutions decay
exponentially if
$$|c|<{|\lambda_1|}.$$
We refer for example to \cite[Chapter 6]{davies} for estimates on $\lambda_1$ and for further references. The same result holds for  more general elliptic operators as considered in \cite[Section 6.3]{davies}.
\end{exa}
\end{section}
\begin{section}{Norm continuity}
In this section, we show that if the operator $(A,D(A))$ in (DE) generates an
immediately norm continuous semigroup on $X$, then, under an
appropriate assumption, the operator matrix associated to the delay equation 
generates an eventually norm continuous semigroup on $\cE$. This fact is
important for the study of the asymptotic behaviour of the solutions.
For convenience, we repeat here the definitions from \cite[Definition II.4.17]{en-na}. A strongly continuous semigroup $\left(T(t)\right)_{t\geq 0}$ on a Banach space $Y$ is called {\it eventually norm continuous}, if there exists $t_0\geq 0$ such that the function $t\mapsto T(t)$ is norm continuous from $(t_0,\infty)$ to $\cL(Y)$. The semigroup is called {\it immediately norm continuous} if $t_0$ can be chosen to be $t_0=0$.

Let $(A,D(A))$ be the generator of a strongly continuous semigroup
\sg on a Banach space $X$ and let $C\in\cL ((D(A),\|\cdot\|_A ) ,X)$. 
Moreover, let
us assume that there exist  $\varepsilon >0$ and a function
$q:(0,\varepsilon ) \rightarrow\RR_+$ such that
\(\ds{\lim_{t\searrow 0}q(t)=0}\) and
\begin{equation}\label{miyadera}
\int_0^t \|C\,T(s)x\|\,ds \leq q(t)\,\|x\| 
\end{equation}
for every $x\in D(A)$ and every $0<t<\varepsilon$. Then we know
from the perturbation theorem of Miyadera-Voigt that $(A+C,D(A))$ generates
a strongly continuous semigroup \sgu on $X$ given
by the Dyson-Phillips series
\begin{equation}\label{dy-ph}
U(t) = \sum_{n=0}^\infty (V^n T)(t),\qquad 
t\geq 0,
\end{equation}
where $V$ is the abstract Volterra operator defined in \cite[Theorem 
III.3.14]{en-na} converging uniformly on compact intervals of $\RR_+$
(see \cite[Corollary III.3.15]{en-na}).  
For $x\in D(A)$ we have 
\begin{equation}\nonumber
(V T)(t)x:=\int_0^t T(t-s)CT(s)x\,ds.
\end{equation}
\begin{theo}\label{cond.ev.norm.cont}
If \sg is norm continuous for $t>\alpha$ and there exists
$n\in\NN$ such that $V^n T$ is norm continuous for $t>0$, then the
perturbed semigroup \sgu is norm continuous for $t>n\alpha$.
\end{theo}
The proof is a straightforward generalization of 
\cite[Theorem 6.1]{na-pi}. We refer to \cite[Corollary 2.7]{diss} for the details.
We now apply this result to the operator
$(\cA ,D(\cA ))$ associated to the delay equation (DE).
As before, we will assume that $(A,D(A))$ generates a
strongly continuous semigroup \sgs on $X$ and that the perturbation $\cB$ satisfies the following condition being slightly stronger than condition (M). There exists
$q:\RR_+\longrightarrow \RR_+$ with $\lim_{t\rightarrow0^+}q(t)=0$ and
\begin{equation}\tag*{(K)}
\qquad\int_0^t \|\Phi (S_s x+T_0 (s)f)\|\,ds
\leq q(t)\left\|\xf\right\|
\end{equation}
for all $\xf\in D(\cA_0 )$ and $t>0$.
We then have that $(\cA ,D(\cA ))$ is a generator by Theorem
\ref{theo.miyadera-voigt}. Moreover, from (\ref{cond.miyadera}) it follows that all the cases in Examples \ref{esempi}  satisfy this condition.\\
We recall that $(\cA_0 ,D(\cA_0 ))$ is the operator defined in
(\ref{delaymatrix0}) and (\ref{delaydomain0}) generating the strongly
continuous semigroup \SGO given by (\ref{unperturbeddelaysemigroup}).
\begin{prop}\label{normcont.delay.sg0}
If \sgs is immediately norm continuous, then \SGO is norm continuous for $t\geq
1$.
\end{prop}
\begin{proof}
Let $t\geq 1$. Then $T_0 (t)=0$ and $\cS_0 (t)=\bigl(\begin{smallmatrix}
                                                S(t) & 0\\
						S_t  & 0
						\end{smallmatrix}\bigr)$.
So it suffices to show that the map $t\mapsto S_t$ from $[1,\infty )$
to $\cL (X,\Lp )$ is norm continuous.\\
For $s, t\geq 1$ we have
\begin{equation*}
\begin{split}
\lim_{s\rightarrow t}\|S_s -S_t\| 
 &=\lim_{s\rightarrow t}\sup_{\|x\|\leq 1}
\left(\int_{-1}^0 \|S(s+\sigma )x-S(t+\sigma )x\|^p \,d\sigma\right)
^{\frac{1}{p}}\\
 &\leq \lim_{s\rightarrow t}\sup_{\|x\|\leq 1}
\bigl(\int_{-1}^0 \|S(s+\sigma )-S(t+\sigma )\|^p \,d\sigma\bigr)
^{\frac{1}{p}}\|x\|\\
&= \lim_{s\rightarrow t}
\bigl(\int_{-1}^0 \|S(s+\sigma )-S(t+\sigma )\|^p \,d\sigma\bigr)
^{\frac{1}{p}},
\end{split}
\end{equation*}
which converges to $0$ as $s$ tends to $t$ 
since \sgs is immediately norm continuous and therefore
uniformly norm continuous on compact intervals.
\end{proof}
We are now ready to apply Theorem \ref{cond.ev.norm.cont}.
\begin{prop}\label{normcont.delay.sg}
If the semigroup \sgs generated by $(A,D(A))$ is immediately norm continuous,
then \SG is norm continuous for $t\geq 1$.
\end{prop}
The same result was already proved by a different technique, and for the special case $\Phi :=\sum_{k=0}^n B_k
\delta_{h_k}$, 
by   Fischer and 
 van Neerven \cite[Proposition 3.5]{FvN}.
\begin{proof}[Proof of Proposition \ref {normcont.delay.sg}]
From Proposition \ref{normcont.delay.sg0}, we have that \SGO is norm continuous
for $t\geq 1$. We now show that $V\cT_0$ is norm continuous for $t\geq 0$. In
fact, for $t\geq 0$ and $\xf\in D(\cA_0 )$ we have
\begin{align*}
V\cT_0 (t)\xf 
 &=\int_0^t \cT_0 (t-s)\cB \cT_0 (s)\xf\,ds\\
 &=\int_0^t \cT_0 (t-s)\begin{pmatrix}
                       0&\Phi\\
		       0&0
		       \end{pmatrix}
            \begin{pmatrix}
                       S(s)x\\
		       S_s x+T_0 (s)f 
		       \end{pmatrix}\,ds\\
 &=\int_0^t \begin{pmatrix}
                       S(t-s)&0\\
		       S_{t-s}&T_0 (t-s)
		       \end{pmatrix}
          \begin{pmatrix}
                       \Phi (S_s x+T_0 (s)f)\\
		       0
		       \end{pmatrix}\,ds\\
 &=\int_0^t \begin{pmatrix}
                       S(t-s)\Phi (S_s x+T_0 (s)f)\\
		       S_{t-s}\Phi (S_s x+T_0 (s)f)
		       \end{pmatrix}\,ds.
\end{align*}
We prove norm continuity of both components separately.
1. Let $t\geq 0$ and $1>h>0$. Then we have
\begin{align*}
 & \left\|\int_0^{t+h}S(t+h-s)\Phi (S_s x+T_0 (s)f)\,ds-
 \int_0^t S(t-s)\Phi (S_s x+T_0 (s)f)\,ds\right\|\\
 &\qquad\leq \left\|\int_t^{t+h}S(t+h-s)\Phi (S_s x+T_0 (s)f)\,ds\right\|\\
 &\qquad\phantom{\leq}+\left\|\int_0^t (S(t+h-s)-S(t-s))\Phi (S_s x+T_0
(s)f)\,ds\right\|\\
 &\qquad\leq \int_0^h \|S(h-s)\|\,\|\Phi (S_{s+t} x+T_0
(s+t)f)\|\,ds\\
 &\qquad\phantom{\leq}+\int_0^t \|S(t+h-s)-S(t-s)\|\,\|\Phi (S_s x+T_0
(s)f)\|\,ds.\\
 &\qquad\leq  \sup_{0\leq r\leq 1}\|S(r)\| q(h)\left\|\cT_0 (t)\xf\right\|\\
 &\qquad\phantom{\leq}
+\int_0^t \|S(t+h-s)-S(t-s)\|\,\|\Phi (S_s x+T_0
(s)f)\|\,ds.
\end{align*}
By condition (K), the Lebesgue dominated convergence theorem and by the immediate norm continuity of \sgs, we have that
$$\sup_{0\leq r\leq 1}\|S(r)\| q(h)\left\|\cT_0 (t)\xf\right\|+
\int_0^t \|S(t+h-s)-S(t-s)\|\,\|\Phi (S_s x+T_0
(s)f)\|\,ds$$
tends to $0$ as $h\rightarrow 0^+$ uniformly in $\xf\in
D(\cA_0 )$, $\left\|\xf\right\|\leq1$. The proof for $h\rightarrow 0^-$ is analogous.\\
Since $D(A)$ is dense in $\cE$, the first component of $V\cT_0$ is immediately norm
continuous.
2. To prove immediate norm continuity of the second component of
$V\cT_0$ one proceeds in a  similar way. We only have to
use the norm continuity of the map $t\mapsto S_t$, which was proved in
Proposition \ref{normcont.delay.sg0}.
Hence, the map $t\mapsto V\cT_0 (t)$ is norm
continuous on $\RR_+$ and 
by Theorem \ref{cond.ev.norm.cont} we have that \SG is norm continuous
for $t\geq 1$.
\end{proof}
Using the stability results obtained in the previous section and the 
spectral mapping theorem for eventually norm continuous semigroups (see \cite[Theorem IV.3.9]{en-na} for the details), we 
can prove the following stability result.
\begin{cor}\label{c.stabilita}
Assume that $A$ generates an immediately norm continuous semigroup, $\omega_0 (A)<0$ and let $\alpha\in(\omega_0 (A),0]$. If
\begin{equation}
\sup_{\omega\in\RR}\left\|\Phi (\epsilon_{\alpha +i\omega}\otimes Id)
\right\|<\frac1{\sup_{\omega\in\RR}\left\|R(\alpha +i\omega ,A)\right\|},
\end{equation}
then $\omega_0 (\cA )<\alpha\leq 0$.
\end{cor}
\begin{exa}
We consider the reaction-diffusion equation from Example \ref{parab}
in the state space $\cE:=L^r(\Omega)\times L^p\left([-1,0],L^r(\Omega)\right)$ for $1\leq r<\infty$, $1<p<\infty$.
The well-posedness follows again from the calculations in Examples \ref{esempi}(b). We show again that the solutions decay exponentially if
\begin{equation}\label{lambda1}
|c|<|\lambda_1|,
\end{equation}
extending our result from the Hilbert space case.
Unfortunately, Corollary \ref{c.stabilita} is not optimal for this problem. If we estimate the resolvent for the stability condition in Corollary \ref{c.stabilita}, we obtain that it is sufficient for the exponential stability if 
$$|c|<c_r,$$
where $c_r\leq|\lambda_1|$ is a constant depending on $r$.
To obtain estimate (\ref{lambda1}), we extend the result obtained for the Hilbert space case in two steps.
Consider first the case $r=2$. We know from Lemma \ref{spectrum} that $\sigma(\cA)$, the spectrum of $(\cA, D(\cA))$, does not depend on $p$, and that
$$s(\cA)=\omega_0(\cA)$$
since by Proposition \ref{normcont.delay.sg} $(\cA, D(\cA))$ generatates an eventually norm continuous semigroup on $\cE$. For 
$p=2$ we had a condition to obtain exponential stability, i.e., $\omega_0(\cA)<0$.
Hence, as in the Hilbert space case, we obtain that the solutions decay exponentially if (\ref{lambda1}) holds.

For the general case $\cE:=L^r(\Omega)\times L^p\left([-1,0],L^r(\Omega)\right)$, we observe first that the operator $(A,D(A))$ has compact resolvent. From \cite[Proposition 1.12(ii)]{en-na} it follows that the operator $\left(A+\Phi (\epsilon_{\lambda}\otimes Id),D(A)\right)$ has compact resolvent in $X:=L^r(\Omega)$ and thus its spectrum does not depend on $r$, see \cite[Proposition 2.6]{arendt} for the details.
Using a spectral mapping argument as before, we obtain  condition (\ref{lambda1}) for the exponential stability.
\end{exa}

\end{section}
%
%
%
%
%
%
%bibliografia
%

\end{document}